\documentclass[11pt]{article}
\usepackage{authblk}

\usepackage{epsfig}
\usepackage{amsmath,amssymb,amsthm,verbatim,pdfsync}
\usepackage{mathtools} 
\usepackage{mathrsfs}   
\usepackage[utf8]{inputenc}
\usepackage{hyperref}

\usepackage{tikz,eucal,enumerate,mathrsfs,bbm,verbatim}

\numberwithin{equation}{section}
\numberwithin{figure}{section}

\newtheorem{thm}{Theorem}[section]
\newtheorem{theorem}[thm]{Theorem}

\newtheorem{prop}[thm]{Proposition}

\newcommand{\beqn}{\begin{equation}}
\newcommand{\eeqn}{\end{equation}}

\def\R{{\mathbb R}}
\def\Db{{\overline{D}}}
\def\Dbar{{\overline{D}}}
\def\Bbar{{\overline{B}}}

\def\F{{\cal F}}

\def\L{{\cal L}}

\def\no{{\nabla_\omega}}
\def\nop{{\nabla_\omega^\perp}}
\def\bo{{b_\omega}}
\def\bop{{b_\omega^\perp}}

\newcommand{\cleq}{\preccurlyeq}
\def\cgeq{\succcurlyeq}

\def\pa{{\partial}}
\def\ep{\epsilon}
\def\nn{\nonumber}

\def\Bb{{\overline{B}}}
\def\xtot{{x,t; \omega, \tau}}
\def\ttxo{{t-\tau-x \cdot \omega}}
\def\xto{{x,t;\omega}}
\def\txo{{t-x \cdot \omega}}

\def\uacu{{\acute{u}}}
\def\ubar{{\bar{u}}}
\def\vacu{{\acute{v}}}
\def\vbar{{\bar{v}}}
\def\wacu{{\acute{w}}}
\def\aacu{{\acute{a}}}
\def\abar{{\bar{a}}}
\def\bacu{{\acute{b}}}
\def\bbar{{\bar{b}}}
\def\cacu{{\acute{c}}}
\def\cbar{{\bar{c}}}
\def\alphaacu{{\acute{\alpha}}}
\def\qacu{{\acute{q}}}
\def\qbar{{\bar{q}}}
\def\Lacu{{\acute{\L}}}
\def\phiacu{{\acute{\phi}}}
\def\psiacu{{\acute{\psi}}}

\title{
Stability for a formally determined inverse problem for a hyperbolic PDE with space and time dependent coefficients}
\author[1]{Venkateswaran P.\ Krishnan}
\author[2]{Rakesh}
\author[3]{Soumen Senapati\footnote{Corresponding author}}
\affil[1]{TIFR Centre for Applicable Mathematics, Sharada Nagar, Chikkabommasandra, Bangalore, Karnataka 560065, India, Email: {\tt vkrishnan@tifrbng.res.in}}
\affil[2]{Department of Mathematical Sciences, University of Delaware,
	501 Ewing Hall,
	Newark, DE 19716, USA, Email: {\tt rakesh@udel.edu}}
\affil[3]{TIFR Centre for Applicable Mathematics, Sharada Nagar, Chikkabommasandra, Bangalore, Karnataka 560065, India, Email: {\tt soumen@tifrbng.res.in}}
\date{}
\begin{document}

\maketitle

\begin{abstract}
We prove stability for a formally determined inverse problem for a hyperbolic PDE in one or higher space dimensions with the 
coefficients dependent on space and time variables. The hyperbolic operator has constant wave speed and we study the 
recovery of the 
first order and zeroth order coefficients. We use a modification of the Bukhgeim-Klibanov method to obtain our results. 
\end{abstract}


\def\tpurple{\textcolor{purple}}
\def\tred{\textcolor{red}}
\def\tblue{\textcolor{blue}}

\section{Introduction}

Suppose $D$ is a bounded domain in $\R^n$, $n \geq 1$, with a smooth boundary and $T>0$. Let 
$a(x,t), c(x,t)$ be smooth real valued functions on $\Db \times [0,T]$ and 
$b(x,t) = (b^1(x,t), \cdots, b^n(x,t))$ a smooth $n$-dimensional 
real vector field on $\Db \times [0,T]$. Define the hyperbolic operator
\begin{align}
	\L_{a,b,c} & := (\pa_t-a)^2  -  (\nabla - b)^2 + c
	\label{eq:Ldef1}
	\\
	& = \Box - 2a \pa_t + 2b \cdot \nabla + c - a_t + \nabla \cdot b + a^2 - b^2.
	\label{eq:Ldef2}
\end{align}
When it is clear from the context, we use $\L$ instead of $\L_{a,b,c}$.

Let $w(x,t)$ be the solution of the well-posed IBVP
\begin{align}
	\L_{a,b,c} w = 0, & \qquad (x,t) \in D \times [0,T]
	\label{eq:Wde}
	\\
	w(\cdot,0) =f, ~~ w_t(\cdot,0) =g, & \qquad \text{on } D
	\label{eq:Wic}
	\\
	w = h, & \qquad \text{on } \pa D \times [0,T]
	\label{eq:Wbc}
\end{align}
for $f,g,h$ with appropriate regularity.
For a given $a,b,c$, define the response operator
\begin{equation}\label{eq:lambda}
	\Lambda_{a,b,c} : (f,g,h) \to \left [ w(\cdot,T)|_D, w_t(\cdot,T)|_D, ~ \pa_\nu w|_{\pa D \times [0,T]} \right ];
\end{equation}
hence $\Lambda_{a,b,c}(f,g,h)$ represents the boundary and final time response, of the acoustic medium with 
acoustic properties $(a,b,c)$, to the initial boundary input $(f,g,h)$. 
So we have the forward map
\[
\Lambda: (a,b,c) \to \Lambda_{a,b,c},
\]
whose injectivity and stability has been studied by several authors.
This is an overdetermined problem (when $n>1$) because 
the distribution kernel of
$\Lambda$ depends on $2n$ parameters while $a,b,c$ depend on $n+1$ parameters.
Our goal is to study the recovery of $a,b,c$ from less (but slightly different) data than $\Lambda_{a,b,c}$ - we study a formally determined problem where the data depends only on $n+1$ parameters. 
Before we state our goal, we first describe what is known about the injectivity and stability of $\Lambda$ type forward maps.

In general, $\Lambda$ is not injective, due to gauge invariance (described later), and in such cases one hopes to recover curl$(a,b)$ and $c$ or one studies special cases when $a,b$ are known or $c$ is known. Below, injectivity and stability results for $\Lambda$ type forward maps
are to be understood in this sense. We use the term $\Lambda$ type 
forward maps because there are results in the literature with one or more of the following:
\vspace{-0.2in}
\begin{itemize}
	\item data is collected only on a part of the lateral boundary
	\item data is not collected on $t=T$
	\item there are no sources on $t=0$
	\item the data is the far field pattern in the frequency domain, which in some sense is equivalent to $\Lambda$ but with $t$ varying over $(-\infty, \infty)$
	\item the principal part of the operator is not the wave operator but a hyperbolic operator associated with a non-constant velocity or even a Lorentzian metric.
\end{itemize}
\vspace{-0.2in}
While the inverse problems associated with $\Lambda$ type forward maps are overdetermined problems, there are considerable challenges dealing with some of 
these problems, either because three coefficients are being determined simultaneously, or the data is given only on a part of the lateral boundary, or the wave velocity is non-constant.
The results we obtain are only for the constant velocity case, though for a formally determined problem. 

From domain of dependence arguments, it is clear that, for hyperbolic operators
with coefficients dependent on $x,t$ and measurements over a finite $t$ interval $[0,T]$, to recover the coefficients on $D \times [0,T]$ one needs sources on $D \times \{t=0\}$ and measurements on $D \times \{t=T\}$, in addition to the lateral boundary sources and measurements. So, for inverse
problems with coefficients dependent on $x$ and $t$, with sources only on the lateral boundary and receivers/measurements only on the lateral boundary, of the $x,t$ domain,
either one must know the coefficients in appropriate regions contiguous with $t=0$ and $t=T$, assume analyticity of the coefficients with respect to $t$, or have data from measurements over infinitely long $t$ intervals.
The situation is different when the principal part of the operator is not the wave operator (or coming from a Lorentzian metric) but the Schr\"odinger operator $i \pa_t + \Delta$ (infinite speed of propagation) or perhaps a fractional differential operator (a non-local operator). We do not describe the results for such operators.

For coefficients which depend on $x,t$, results on the injectivity of $\Lambda$ type forward maps, for data on infinite time intervals, may be found in, for example,
\cite{Ramm_Sjostrand_IP_wave_equation_potential_1991,Stefanov1987,Stefanov_Inverse_scattering_potential_time_dependent_1989,
	Salazar_time-dependent_first_order_perturbation_2013}. For the finite time interval case, the injectivity of $\Lambda$ type forward maps,
but with coefficients known in certain regions near $t=0$ and $t=T$ or analytic in $t$, results may be found in, for example,
\cite{Isakov_Completeness_Product_solutions_IP_1991, Ramm_Rakesh_Property_C_1991,
	Eskin_IHP_time-dependent_2007,
	Eskin_IP_hyperbolic_PDE_time-dependent_2017,
	Bellassoued_Ben_Uniqueness_IP_time_dependent_coefficient_2016,
	Bellassoued-BenAicha, 
	Kian_Damping_partial_data_2016, 
	Krishnan-Vashisth, 
	Stefanov-Yang, FIKO}.
The stability of $\Lambda$ has been studied extensively in, for example, 
\cite{ Stefanov_Uhlmann_Stability_estimate_hyperbolic_DN_Map_1998,
	Salazar_stability_first_order_2014,Bellassoued_Jellali_Yamamoto_Stability_estimate_Inverse_BVP_local_DN_Map_2008,
	senapati2020stability, BB, BenAicha2015stability}. 
The results mentioned here, for $x,t$ dependent coefficients, are for over-determined problems and the stability results, even for these over-determined problems, are of 
log-log type. There are better stability results for the Schr\"odinger operator (infinite speed of propagation) with Holder stability (but not Lipschitz stability), still for an over-determined problem - see \cite{KS}.

We do not survey results for $\Lambda$ type maps when the 
coefficients are independent of $t$ - no sources are needed on $t=0$ and no measurements are needed on $t=T$. 
A brief survey of such results may be found in 
\cite{KKLO}. Most of these results
use generalizations of the Boundary Control Method introduced by Belishev 
(see  \cite{Belishev_Recent_progress_BC_Method_2007, Belishev_BC_Method_2011}) 
or generalizations of geometric optics solutions for hyperbolic PDEs introduced in 
\cite{Rakesh_Symes_Uniqueness_1988}, which were themselves imitations of similar (but 
harder to construct) solutions for elliptic PDEs constructed by Sylvester and Uhlmann in \cite{Sylvester_Uhlmann_Calderon_problem_1987}.

We now describe results for formally determined inverse problems for hyperbolic PDEs.

For coefficients independent of $t$, there are uniqueness and stability results, for formally determined problems, based on the ideas introduced by Bukhgeim and Klibanov in \cite{Bukhgeuim_Klibanov_Uniqueness_1981} which had the first such results in dimension
$n>1$. See \cite{Bellassoued_Yamamoto_Carleman_book} for a survey of such results and an exposition of the significant modifications of the important ideas in  \cite{Bukhgeuim_Klibanov_Uniqueness_1981}. The only drawback of these results is that they require the {\em initial source} to be a positive (or negative) function throughout the domain (in $x$ space).
Rakesh and Salo, in \cite{RS1,rakesh2019fixed}, obtained uniqueness and stability for the case
$a=0, b=0$ (recover $c$), avoiding the use of positive initial sources, using instead the more natural incoming plane wave source, except one needed
data from two such experiments, corresponding to incoming plane waves coming from opposite directions. In \cite{ma2020fixed, MPS2020}, these ideas were extended to obtain similar results for
the operator with general $a,b,c$  or the operator associated with a 
Lorentzian metric (with restrictions).

We also note the work in \cite{lassas2020uniqueness}, on a coefficient recovery problem for a semilinear hyperbolic PDE, with the coefficient independent of $t$ and the data  consisting of a weighted average of lateral boundary measurements. This seems to be an under-determined inverse problem but the non-linearity of the PDE is crucial for this result. The article \cite{feizmohammadi-Kian2020} also contains a uniqueness (and reconstruction) result for a formally determined $a,b,c$ recovery problem with the coefficients dependent on $x$ and $t$. They use a single boundary source $h$, constructed as the infinite sum of a combination of sources, each generating a solution travelling along a ray for the hyperbolic PDE, and the rays associated with these solutions forming a dense subset of the $x,t$ domain. The challenge is to build the source $h$ so that the data from the $h$ source can be separated into the data contributions from the sources in the sum. We believe such a source $h$ on the lateral boundary
would have support consisting of the full lateral boundary. 

The articles \cite{RS1,rakesh2019fixed} were attempts at (and have come close to) solving the long-standing open {\em Fixed Angle Scattering} inverse problem. There are other long-standing formally determined open problems for hyperbolic PDEs 
(with coefficients independent of $t$) such as the {\em Back-scattering Problem}, where the results are much weaker than the result for the {\em Fixed Angle Scattering Problem}. We do not survey the results for these two problems as the introductions to
\cite{MPS2020, RS1, RU} have a good survey of the results.

We study a formally determined inverse problem with the coefficients $a,b,c$ dependent on $x,t$. We prove uniqueness (up to gauge) and Lipschitz stabilty using modifications of the ideas of Bukhgeim and Klibanov in \cite{Bukhgeuim_Klibanov_Uniqueness_1981}, of an idea in \cite{MPS2020} and our new idea for problems with coefficients dependent on $x,t$. Our results have one weakness - the problem must be posed in the full space
$\R^n \times (-\infty, T]$ and do not work for space-time cylinders with bounded bases such
as $D \times (-\infty, T]$

Let $B$ denote the open unit ball in $\R^n$, $n \geq 1$, $T>0$ and suppose
$a(x,t), b^i(x,t), c(x,t)$, $i=1, \cdots, n$ are compactly supported smooth functions on $\R^n \times \R$.
If $\omega$ is a unit vector in $\R^n$ and $\tau \in \R$, let $U(\xtot)$ be the solution of the IVP
\begin{align}
	\L U = 0, & \qquad \text{on } \R^n \times \R,
	\label{eq:Ude}
	\\
	U(\xtot) = H(t-\tau - x \cdot \omega), & \qquad x\in \R^n, ~ t \ll 0,
	\label{eq:Uic}
\end{align}
and let $V(\xtot)$ be the solution of the IVP
\begin{align}
	\L V = 0, & \qquad \text{on } \R^n \times \R,
	\label{eq:Vde}
	\\
	V(\xtot) = \delta(t-\tau - x \cdot \omega), & \qquad x \in \R^n, ~ t \ll 0.
	\label{eq:Vic}
\end{align}
So $U,V$ are the disturbances in the medium caused by two types of impulsive incoming plane waves.
Here $\tau$ is the time the incoming plane wave reaches the origin; $\tau$ may also be regarded as a time delay.
Given $T>0$, define the map
\[
\F : (a,b,c) \to [ U, U_t, V, V_t](x,T;\omega,\tau)|_{x \in \R^n, \omega \in \Omega, \tau \in (-\infty, T+1]}
\]
mapping the medium properties $(a,b,c)$ of the region $\R^n \times (-\infty,T]$, to the final time medium response, to incoming plane waves, 
coming from a finite set of directions $\omega $ in the finite set 
$\Omega$ of unit 
vectors in $\R^n$, with delays $\tau \in (-\infty,T+1]$. Our goal is to study the injectivity and stability of $\F$.
We note the data set for our inverse problem (associated with $\F$) 
depends on $n+1$ parameters and our unknown functions
$(a,b,c)$ depend on $n+1$ parameters; hence our problem is formally determined.

We introduce definitions used throughout the article.
Given a unit vector $\omega \in \R^n$, a $\tau \in \R$ and a $T>0$, we define the wedge shaped
region (see Figure \ref{fig:qwt})
\begin{figure}[h]
\begin{center}
\epsfig{file=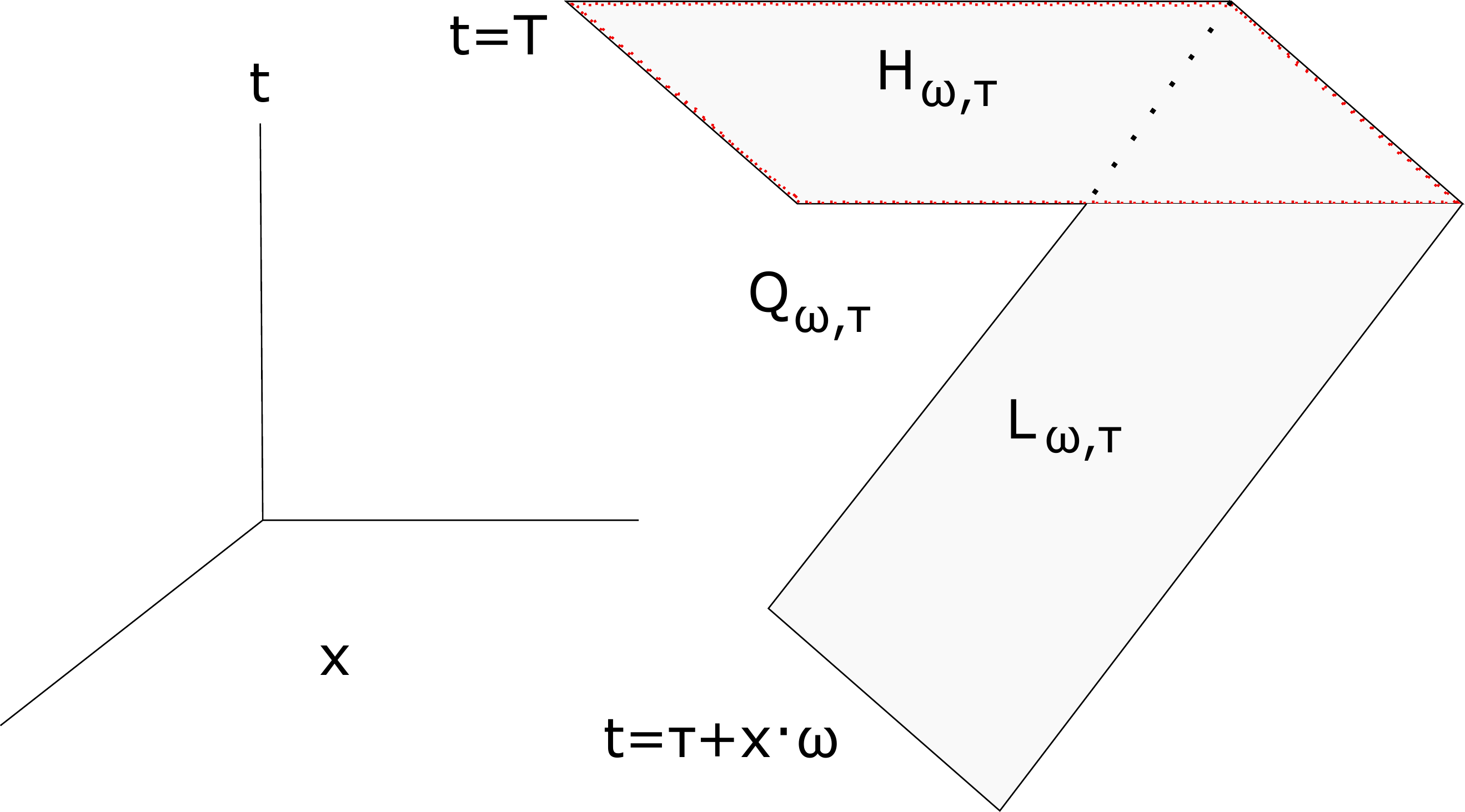, height=1.5in}
\end{center}
\caption{The wedge shaped region and its boundary}
\label{fig:qwt}
\end{figure}
\[
Q_{\omega, \tau} = \{(x,t) \in \R^n \times \R :  \tau + x \cdot \omega \leq t \leq T\}
\]
and its higher and lower boundary
\[
H_{\omega, \tau} = Q_{\omega,\tau} \cap \{t=T \}
\qquad 
L_{\omega,\tau} = Q_{\omega,\tau} \cap \{t= \tau + x \cdot \omega \}.
\]
We suppress the $T$ dependence of these sets as $T$ will not vary.
For any submanifold $M$ of $\R^n \times \R$ and a function $w$ on $M$ we define the weighted norms
\[
\|w\|_{1,M,\sigma} = \left ( \int_M e^{2 \sigma t} (|\nabla_M w|^2 + \sigma^2 |w|^2 ) \right)^{1/2},
\qquad 
\|w\|_{0,M,\sigma} = \left ( \int_M e^{2 \sigma t} |w|^2 \right)^{1/2},
\]
where $\nabla_M$ is the gradient on the manifold $M$ made up only of derivatives in directions 
tangential to $M$. We will also use $\|w\|_{1,M}$, $\|w\|_{0,M}$ for the standard $H^1$ and $L^2$ norms
on $M$.

Given compactly supported smooth functions $a,b^j,c$ on $\R^n \times \R$, we define the function
\beqn
\alpha(x,t;\omega) := \exp \left( \int_{-\infty}^0 (a+ \omega \cdot b)(x + s \omega, t + s) \, ds \right ),
\qquad (x,t) \in \R^n \times \R.
\label{eq:alphadef}
\eeqn
Note that $\alpha(x_0,t_0; \omega)$ is determined by the values of $a,b$ in the region $t \leq t_0$.

We start with the well-posedness of the IVP associated with $U$ and $V$. 
%
%
\begin{prop}[The Heaviside function solution]\label{prop:heaviside}
Suppose $a,b^i,c$, $i=1, \cdots, n$, are compactly supported smooth functions on $\R^n \times \R$, 
$\omega$ a unit vector in $\R^n$ and $\tau \in \R$. The IVP (\ref{eq:Ude}) - (\ref{eq:Uic}) has a unique distributional solution
\[
U(x,t;\omega,\tau) = u(x,t;\omega,\tau) H(t-\tau - x \cdot \omega), \qquad (x,t) \in \R^n \times \R
\]
where $u(x,t;\omega, \tau)$ is a smooth function in the region $t \geq \tau + x \cdot \omega$ and is the unique solution of the characteristic IBVP
\begin{align}
\L_{a,b,c} u =0, & \qquad x \in \R^n, ~ \tau + x \cdot \omega \leq t,
\label{eq:ude}
\\
u(\xtot) = \alpha(x,t; \omega), & \qquad x \in \R^n, ~ t = \tau + x \cdot \omega,
\label{eq:ucc}
\\
u(\xtot) =1, & \qquad   x \in \R^n, ~ \tau + x \cdot \omega \leq t \ll 0.
\label{eq:uic}
\end{align}
Further, given $T>0$ if $ \|[a,b,c]\|_{C^N(Q_{\omega, \tau})} \leq M$ for $N = 5 + [n/2]$ then 
\[
\| u\|_{C^3(Q_{\omega,\tau})} \leq C,
\]
where $C$ depends on $M$ and the support of $a,b,c$.
\end{prop}

A similar result is true for $V(x,t;\omega,\tau)$. 
%
\begin{prop}[The delta function solution]\label{prop:delta}
Suppose $a,b^i,c$, $i=1, \cdots, n$, are compactly supported smooth functions on $\R^n \times \R$, 
$\omega$ a unit vector in $\R^n$ and $\tau \in \R$. The IVP (\ref{eq:Vde}) - (\ref{eq:Vic}) has a unique distributional solution
\[
V(x,t;\omega,\tau) =  \alpha(\xtot) \, \delta(\ttxo) + v(\xtot) H(\ttxo), \qquad (x,t) \in \R^n \times \R
\]
where $v(x,t;\omega, \tau)$ is a smooth function on the region $t \geq \tau + x \cdot \omega$ and is the unique solution of the characteristic IBVP
\begin{align}
\L_{a,b,c} v &=0,  \qquad t \geq \tau + x \cdot \omega,
\label{eq:vde}
\\
v(x,t;\omega, \tau) & =0, \qquad t\ll 0
\label{eq:vic}
\\
v_t + \omega \cdot \nabla v - (a + \omega  \cdot b) v& = 
- \frac{1}{2} \L_{a,b,c} \alpha,
 \qquad t = \tau + x \cdot \omega.
\label{eq:vcc}
\end{align}
Further, given $T>0$ if $ \|[a,b,c]\|_{C^N(Q_{\omega, \tau})} \leq M$ for $N = 7 + [n/2]$ then 
\[
\| v\|_{C^3(Q_{\omega,\tau})} \leq C,
\]
where $C$ depends on $M$ and the support of $a,b,c$.
\end{prop}
While $V = -\pa_\tau U$, the relationship between $u$ and $v$ may be a little more complicated
because the domains of $u,v$ depend on $\tau$.

The inverse problem has a gauge invariance. If $\phi(x,t)$ is a smooth function on $\R^n \times \R$ then, 
for any smooth function $f(x,t)$ on $\R^n \times \R$, we have
\beqn
(\pa_t -a - \phi_t )(e^\phi f) = e^{\phi} (\pa_t - a) f,
\qquad (\nabla - b - \nabla \phi) (e^\phi f) = e^{\phi} (\nabla - b)f
\label{eq:gauge1}
\eeqn
implying
\beqn
\L_{a + \phi_t, b+ \nabla \phi, c} (e^\phi f) = e^\phi \L_{a,b,c} f;
\eeqn
in particular
\[
\L_{a + \phi_t, b+ \nabla \phi, c} (e^\phi U) = e^\phi \L_{a,b,c} U=0,
\qquad
\L_{a + \phi_t, b+ \nabla \phi, c} (e^\phi V) = e^\phi \L_{a,b,c} V =0,
\]
Hence, if $\phi$ is compactly supported  then $e^\phi U$ and $e^{\phi}V$ are the Heavisde function and delta function solutions
corresponding to the triple $(a+ \phi_t, b + \nabla \phi, c)$. So, if we also have $\phi(\cdot,T) =0$, then
\[
\F(a,b,c) = \F(a+ \phi_t, b + \nabla \phi, c).
\]
Actually our data on $t=T$ will also involve time derivatives of $U,V$ so, for gauge invariance, we will also
need some time derivatives of $\phi$ to be zero at $t=T$. We will be specific below.

We state our principal results next. We have seen in (\ref{eq:Ldef2}) that $\L_{a,b,c}$ can also be written in 
the form
\[
\L_{a,b,c} = \Box - 2a \pa_t + 2 b \cdot \nabla + q
\]
where
\beqn
q := c - a_t + \nabla \cdot b + a^2 - b^2.
\label{eq:Ldef}
\eeqn
We can regard the operator $\L_{a,b,c}$ as determined by the functions $a, b^i, c$ or by the functions
$a, b^i, q$. We use both points of view below - the context will clarify the point of view in play.

Our work has two new ideas, perhaps one more significant than the other. Our most significant idea
allows us to obtain Lipschitz stability for a formally determined 
$x,t$ dependent coefficient problem as compared to the logarithmic stability results for overdetermined 
problems (though on bounded domains) in the literature.
This is showcased in its simplest form in the study of the less
complicated problem of recovering $q$ given $a,b$. Our second idea is about separating the estimates 
on $c$ from the estimates on $a,b$ when we prove stability for the $a,b,c$ problem.

We start with the stability result about recovering $q$, given $a,b$. 
\begin{theorem}[Stability for the $q$ recovery problem, given $a,b$]\label{thm:qstability}
Suppose $T>0$ and $a(x,t), b^i(x,t)$, $i=1, \cdots,n$,
 are compactly supported smooth functions on $\R^n \times [0,T]$ and
$\omega$ is a unit vector in $\R^n$.
If $q, \qacu$ are compactly supported smooth functions on $\R^n \times (-\infty,T]$ with support
in $\Bbar \times [0,T]$ and $\|[q, \qacu, a,b]\|_{C^{7+[n/2]}} \leq M$ then
\[
\| q - \qacu \|_{L^2} 
\cleq 
\int_{-1}^{T+1} \| (v - \vacu)(\cdot, T; \omega, \tau) \|_{1, H_{\omega, \tau}}
+   \| (v_t - \vacu_t)(\cdot, T; \omega, \tau) \|_{0, H_{\omega, \tau}}
~ d \tau.
\]
Here $v, \vacu$ are the functions associated with $(a,b,q)$ and $(a,b,\qacu)$ in Proposition \ref{prop:delta} and the constant depends on $M$ and 
the support of $a,b,q, \qacu$.
\end{theorem}
The proof of this theorem presents one of our ideas, uncluttered by the complications appearing in the proofs of the other theorems.

Next we state a stability result about recovering $a,b$ if $q$ is known. Note there is no gauge invariance if 
$q$ is known. Below $e^1, \cdots, e^n$ is the standard basis for $\R^n$.
%
\begin{theorem}[Stability for the $a,b$ recovery problem, given $q$]\label{thm:abstability}
Suppose $T>0$ and $q(x,t)$ is a smooth compactly supported smooth functions on $\R^n \times [0,T]$.
If $a,b, \aacu, \bacu$ are compactly supported smooth functions on $\R^n \times (-\infty,T]$ with support
in $\Bbar \times [0,T]$ and $\|[a,b, q, \aacu, \bacu, \qacu]\|_{C^{7+[n/2]}} \leq M$ then
\begin{align*}
\| [a-\aacu, b - \bacu] \|_{L^2} 
& \cleq 
\sum_\omega
\int_{-1}^{T+1} \| (u - \uacu)(\cdot, T;  \omega, \tau) \|_{1, H_{\omega, \tau}}
 +    \| (u_t - \uacu_t)(\cdot, T;  \omega, \tau) \|_{0, H_{\omega, \tau}}
\, d \tau
\end{align*}
where  $\omega$ takes the values $-e^n$ and $e^1, \cdots, e^n$.
Here $u, \uacu$ are the functions associated with
$(a,b,q)$ and $(\aacu, \bacu,q)$ in Proposition \ref{prop:heaviside}. The constant depends on $M$ and
the supports of $a,b, \aacu, \bacu, q$.
\end{theorem}

Next we have a uniqueness result about recovering $(a,b,c)$. Noting the gauge invariance mentioned earlier
in the introduction, the most we can hope to recover is curl$(a,b)$ and $c$. However, for $\phi(x,t)$ to 
be a gauge we needed  $\phi(\cdot, T)=0$ (and $\phi_t(\cdot,T)=0$ because of the data we use in our theorems) - this is reflected in the hypothesis 
of the next theorem.
%
\begin{theorem}[Uniqueness for the curl$(a,b)$ and $c$ recovery problem]\label{thm:abcunique}
Suppose $T>0$ and $a,b,c,  \aacu, \bacu, \cacu$ are compactly supported smooth functions on 
$\R^n \times (-\infty,T]$ with support in $\Bbar \times [0,T]$.  If
\begin{align*}
[u, u_t](x,T,\omega, \tau) = [\uacu, \uacu_t](x,T;\omega, \tau),
		& \qquad \forall x \in H_{\omega,\tau}, ~ \tau \in [-1,T+1],  ~\omega = e^i,\pm e^n, ~ i=1, \cdots,  n-1,
\\
[v, v_t](x,T, e^n, \tau) = [\vacu, \vacu_t](x,T; e^n, \tau),
& \qquad \forall x \in H_{\omega,\tau}, ~ \tau \in [-1,T+1], 
\end{align*}
and
\begin{align*}
\int_{-\infty}^T (a + b^n)( x + s e^n, s) \, ds & = \int_{-\infty}^T (\aacu + \bacu^n)( x + s e^n, s) \, ds,
\qquad \forall x \in \R^n
\\
\int_{-\infty}^T (a_t + b_t^n)( x + s e^n, s) \, ds & = \int_{-\infty}^T (\aacu_t + \bacu_t^n)( x + s e^n, s) \, ds,
\qquad \forall x \in \R^n,
\end{align*}
then
\[
c= \cacu, \qquad d \left ( a dt + \sum_{i=1}^n b^i dx^i  \right) 
= d \left( \aacu dt + \sum_{i=1}^n \bacu^i dx^i \right ).
\]
Here $u, v, \uacu, \vacu$ are the functions associated with
$(a,b,c)$ and $(\aacu, \bacu, \cacu)$ in Propositions \ref{prop:heaviside} and \ref{prop:delta}.
\end{theorem}
%
This result is obtained by combining our most significant idea with an idea in \cite{MPS2020} about a 
uniqueness problem 
for a {\em time independent} coefficient determination problem. We do not know how to prove
a similar uniqueness result when all the three coefficients $a,b,q$ are to be recovered - that problem does not
have gauge invariance.

Our final result is a stability result for the $(a,b,c)$ recovery problem. Again, due to the gauge invariance,
we can only expect to recover curl$(a,b)$ and $c$. To obtain stability we need more data than was 
needed for the uniqueness result in Theorem \ref{thm:abcunique}. We define $\psi(x,t)$ to
be the solution of the IVP
\begin{align}
\Box \psi =\nabla \cdot b - a_t + c & \qquad (x,t) \in \R^n \times (-\infty,T]
\label{eq:pside}
\\
\psi(\cdot,t) = 0, & \qquad  t \ll 0.
\label{eq:psiic}
\end{align}

\begin{theorem}[Stability for curl$(a,b)$ and $c$ recovery problem] \label{thm:abcstability}
Suppose $T>0$ and $a,b,c,  \aacu, \bacu, \cacu$ are compactly supported smooth functions on 
$\R^n \times (-\infty,T]$ with support in $\Bbar \times [0,T]$. 
\\
If 
$\|[a,b,c,\aacu, \bacu, \cacu]\|_{C^{7 + [n/2]}} \leq M$
then
\begin{align*}
\| [c - \cacu, d \eta - d \acute{\eta} ]\|_{L^2}
& \cleq
\sum_\omega \int_{-1}^{T+1} 
\| (u - \uacu)(\cdot, T; \omega, \tau) \|_{2, H_{\omega, \tau}}
 +    \| (u_t - \uacu_t)(\cdot, T; \omega, \tau) \|_{1, H_{\omega, \tau}}
 \, d \tau
 \\
& \qquad  + 
 \sum_\omega \int_{-1}^{T+1} 
  \| (u_{tt}- \uacu_{tt})(\cdot, T; \omega, \tau) \|_{0, H_{\omega,\tau}}
\, d \tau
\\
& \qquad + \sum_\omega \int_{-1}^{T+1} 
\| (v - \vacu)(\cdot, T; \omega, \tau) \|_{1, H_{\omega, \tau}}
 +    \| (v_t - \vacu_t)(\cdot, T; \omega, \tau) \|_{0, H_{\omega, \tau}}
 \, d \tau
\\
& \qquad
+ \|(\psi - \psiacu)(\cdot,T)\|_{2,\R^n} + \|\pa_t (\psi - \psiacu)(\cdot,T)\|_{1,\R^n}
+ \|\pa_t^2 (\psi - \psiacu)(\cdot,T)\|_{0,\R^n},
\end{align*}
with $\omega$ taking the values $e^i,\pm e^n$, $i=1, \cdots,  n-1$.
Here
\[
\eta = a dt + \sum_{i=1}^n b^i d x^i,
\qquad
\acute{\eta} = \aacu dt + \sum_{i=1}^n \bacu^i d x^i,
\]
$u, v, \uacu, \vacu$ are the functions associated with
$(a,b,c)$ and $(\aacu, \bacu, \cacu)$ in Propositions \ref{prop:heaviside} and \ref{prop:delta}
and the constant depends on $M$ and the supports of $a,b,c, \aacu, \bacu, \cacu$.
\end{theorem}

The information about $\psi$ is needed for the stability result in Theorem \ref{thm:abcstability}. This information corresponds to having (for 
odd $n$) the integral of 
$\nabla \cdot b - a_t + c $ on all light cones with vertices on $t=T$ and related quantities.
For the even $n$ case, it would be a weighted integral on such solid cones.

The above theorems used the traces on $t=T$ of $u,v$ and their time derivatives,  for $\tau \in [-1,T+1]$. 
There is no information about $[a,b,c]$ in $u(\cdot, \cdot, \cdot, \tau)$ and $v(\cdot, \cdot, \cdot, \tau)$ 
for $\tau <-1$ because $u,v$ are zero since the support of $[a,b,c]$ is in $\Bbar \times [0,T]$.

The Carleman estimate with explicit boundary terms in Proposition \ref{prop:carl} (in Section \ref{sec:carl}) plays in important role in the 
proofs of the theorems. It is perhaps of mild interest that one can use the weight $t$ in the Carleman
estimate for the wave operator even though this weight is not strongly pseudo-convex. The proofs of our theorems do not require this particular 
weight; any increasing function of $t$, such as the traditional Carleman weight $e^{\lambda t}$ for some large $\lambda$, would be sufficient 
for use in our theorems. 


\begin{figure}[h]
\begin{center}
\epsfig{file=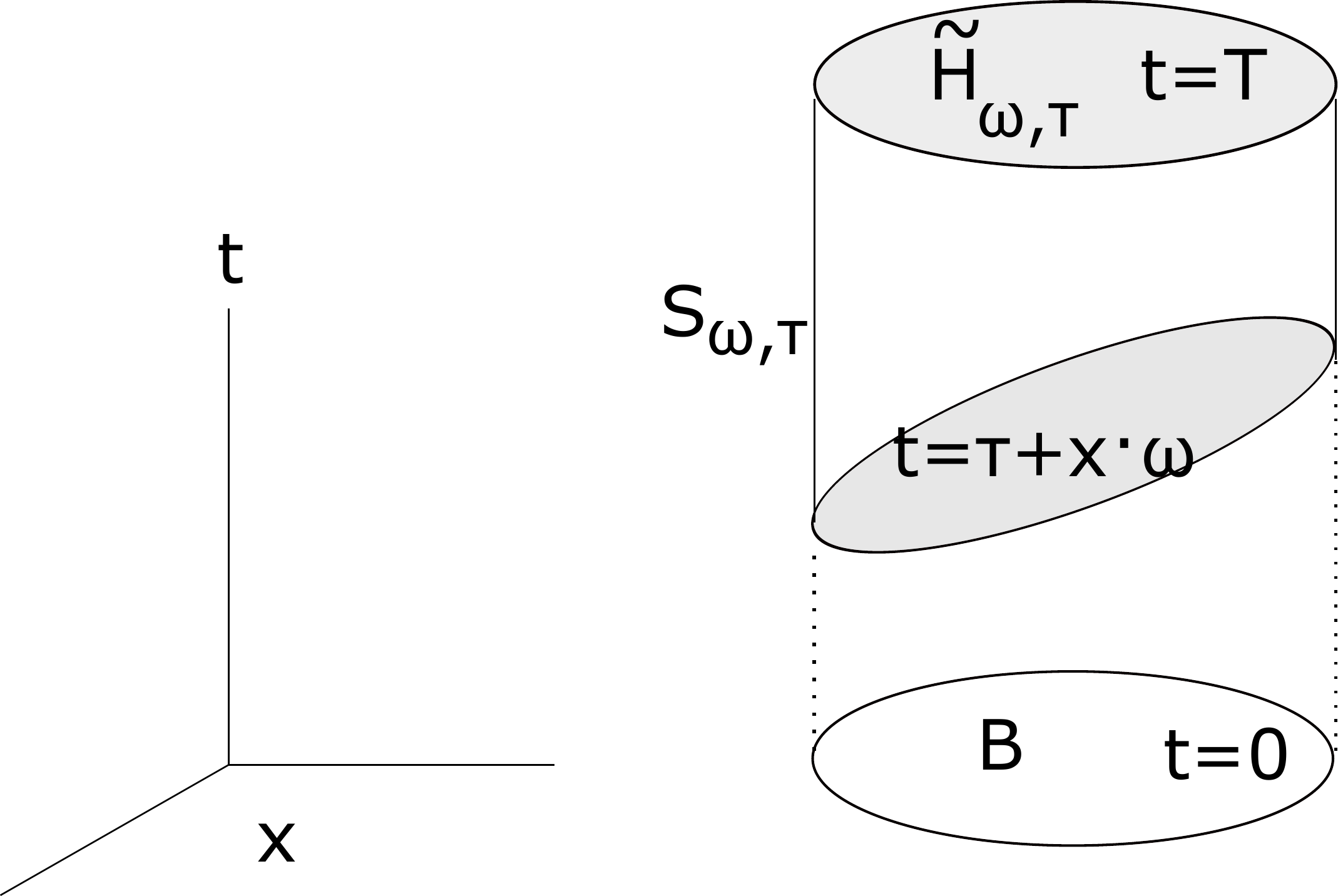, height=1.6in}
\end{center}
\caption{The cylindrical domain and its boundary}
\label{fig:qbwt}
\end{figure}

We can obtain similar results if our data consists of the lateral boundary trace and final time trace on a 
bounded domain, 
that is, if we study the injectivity  and stability of the map
\[
 (a,b,c) \to \left \{ \left [ \pa_{x,t}^\beta u,\pa_{x,t}^\beta v \right ]_{\tilde{H}_{\omega,\tau}},  \,
\left  [\pa_{x,t}^\beta u, \pa_{x,t}^\beta v \right ]_{S_{\omega,\tau}} \right \}_{\omega \in \Omega, \, \tau \in (-\infty, T+1], \, |\beta| \leq 2 }
\]
where $\Omega =\{ \pm e^i : i=1, \cdots, n \}$  and (see Figure \ref{fig:qbwt})
\[
\tilde{H}_{\omega,\tau} =  \left( \Bbar \times \{t=T\} \right) \cap Q_{\omega,\tau}
\qquad
S_{\omega,\tau} = \left( \pa B  \times (-\infty,T] \right )\cap Q_{\omega,\tau}.
\]
To accomplish this we would replace the Carleman estimate for the region $Q_{\omega, \tau}$ in Proposition
\ref{prop:carl} by a Carleman estimate for the region $ (\Bb \times \R) \cap Q_{\omega,\tau}$ and the revised proofs would be almost 
identical to the proofs in this article. The proof of the modified
Carleman estimate also would be almost identical to the proof of Proposition \ref{prop:carl}.

The one weakness of our results is that we {\em cannot} adapt our method to the situation where 
the forward problem is over a bounded domain $D \times (-\infty,T]$ rather than over the free space 
$\R^n \times (-\infty,T]$.
%

We introduce definitions used throughout the article. 
We define the differences
\[
\ubar = u - \uacu, ~ \vbar := v - \vacu, ~~ \abar := a - \aacu, ~~ \bbar := b - \bacu,  ~~\cbar = c - \cacu,
~~ \qbar = q- \qacu.
\]
Sometimes we suppress writing the $a,b,c$ dependence of $\L_{a,b,c}$ and just use $\L$ and $\Lacu$ where
$\Lacu$ corresponds to $\aacu, \bacu, \cacu$. We also have the corresponding functions $\alpha$ and $\alphaacu$ defined in (\ref{eq:alphadef}).

We also note that
\beqn
(\pa_t + \omega \cdot \nabla - (a+ \omega \cdot b) )\alpha(x,t;\omega) =0
\label{eq:alpharel}
\eeqn
as seen from
\begin{align*}
\alpha^{-1} (\alpha_t + \omega \cdot \nabla \alpha)(x,t;\omega) 
& =  \int_{-\infty}^0  
( (\pa_t + \omega \cdot \nabla )(a + \omega \cdot b))( x + s \omega, t+s) \, ds 
\\
& = \int_{-\infty}^0 \frac{d}{ds} \left ( (a + \omega \cdot b)( x + s \omega, t+s) \right ) \, ds
\\
& = (a+ \omega \cdot b)(x,t).
\end{align*}



\section{Proof of Theorem \ref{thm:qstability}}

In this theorem $a=\aacu$, $b=\bacu$. Since $\omega$ is fixed, we suppress the dependence on 
$\omega$.

Using (\ref{eq:vde}), (\ref{eq:vcc}), its version for $\aacu, \bacu, \cacu$ and that $a=\aacu$, $b=\bacu$,
the function $\vbar$ satisfies
\begin{align*}
\L \vbar = - \qbar \vacu, &  \qquad \text{on}~~ Q_\tau,
\\
\vbar = 0, & \qquad t\ll 0,
\\
2(\pa_t + \omega \cdot \nabla - (a + \omega \cdot b))\vbar  = - \qbar \alpha,
& \qquad \text{on  } L_\tau.
\end{align*}
Applying the Carleman estimate in Proposition \ref{prop:carl} to $\vbar$ on the region $Q_\tau$ we have
\[
\sigma \| \vbar\|_{1,\sigma,L_\tau}^2 
 \cleq \| \L \vbar \|_{0,\sigma,Q_\tau}^2 
+ \sigma \|\vbar\|_{1,\sigma,H_\tau}^2 + \sigma \|\pa_t \vbar\|_{0,\sigma,H_\tau}^2.
\]
Since $\alpha$ is positive and bounded 
away from $0$, on 
$L_\tau$ we have
\begin{align*}
|\qbar| \cleq |\qbar \alpha|
 = 2|(\pa_t + \omega \cdot \nabla - (a + \omega \cdot b))\vbar|
\cleq |\vbar_t| + |\nabla \vbar| + |\vbar|
\end{align*}
while on $Q_\tau$
\[
|\L \vbar | = |\qbar \vacu| \cleq |\qbar|,
\]
hence
\beqn
\sigma \| \qbar \|^2_{0,\sigma,L_\tau} 
\cleq \| \qbar\|^2_{0,\sigma,Q_{\tau}}
+ \sigma \|\vbar\|_{1,\sigma,H_\tau}^2 +  \sigma \|\pa_t \vbar\|_{0,\sigma,H_\tau}^2.
\label{eq:carlvbar}
\eeqn

We integrate (\ref{eq:carlvbar}) w.r.t $\tau$ over $[-1, T+1)$. Noting that $\qbar$ is supported
on $\Bbar \times [0,T]$, we define $\qbar=0$ for $t >T$ for convenience
and $L_\tau$ to be the set $t= \tau + x \cdot \omega \leq T$. The LHS  of (\ref{eq:carlvbar}) is
\begin{align*}
\sigma \int_{-1}^{T+1} \int_{\R^n, t = \tau + x \cdot \omega} e^{2 \sigma t} \, |\qbar(x,t)|^2 \, dx \, d \tau
& = \sigma \int_\R \int_{\R^n \times \R} e^{2 \sigma t} \, |\qbar(x,t)|^2
\delta(t-\tau - x \cdot \omega) \, dx \, dt \, d \tau
\\
& = \sigma \int_{\R^n \times \R} \int_\R e^{2 \sigma t} \, |\qbar(x,t)|^2
\delta(t-\tau - x \cdot \omega) \, d \tau \, dx \, dt 
\\
& = \sigma \int_{\R^n \times \R} e^{2 \sigma t} \, |\qbar(x,t)|^2 \, dx \, dt 
\\
& = \sigma \|\qbar\|^2_{0,\sigma, \Bbar \times [0,T]}.
\end{align*}
The integral w.r.t $\tau$ over $[-1, T+1]$, of the RHS of (\ref{eq:carlvbar}), consists of the `data part'
\[
\text{`data'} = \sigma  \int_{-1}^{T+1}  \|\vbar\|_{1,\sigma,H_\tau}^2 +  \sigma \|\pa_t \vbar\|_{0,\sigma,H_\tau}^2
\, d \tau
\]
and (using the support of $\qbar$)
\begin{align*}
\int_{-1}^{T+1} \|\qbar \|_{0, \sigma, Q_\tau}^2
\leq \int_{-1}^{T+1} \|\qbar \|_{0, \sigma, \Bbar \times [0,T]}^2
\cleq \|\qbar\|^2_{0,\sigma, \Bbar \times [0,T]}.
\end{align*}
Combining the two pieces we have
\[
\sigma \|\qbar\|^2_{0,\sigma, \Bbar \times [0,T]}
\cleq \|\qbar\|^2_{0,\sigma, \Bbar \times [0,T]} + \text{`data'}
\]
which gives us the estimate in Theorem \ref{thm:qstability} if we choose $\sigma$ large enough.
%
%
%

\section{Proof of Theorem \ref{thm:abstability}}

Here $q=\qacu$.
The proof is similar to the proof of theorem \ref{thm:qstability} except one uses the solution $U$.

We start with an intermediate estimate for a fixed $\omega,\tau$. We suppress the dependence on
$\omega, \tau$ during the derivation of this intermediate estimate.
Using (\ref{eq:ude}), (\ref{eq:ucc}) and their analogs for $\aacu, \bacu$ and that $q = \qacu$,
$\ubar$ satisfies
\begin{align*}
\L \ubar  = 2 \abar \uacu_t - 2 \bbar \cdot \nabla \uacu,
& \qquad \text{on } Q,
\\
\ubar =0, & \qquad t\ll 0,
\\
\ubar  = \alpha - \alphaacu, & \qquad \text{on } L.
\end{align*}
Applying the Carleman estimate in Proposition \ref{prop:carl} to $\ubar$ on the region $Q$,
we obtain
\begin{align}
\sigma \| \ubar \|_{1,\sigma, L}^2
& \cleq \|[\abar, \bbar] \|_{0,\sigma,Q}^2 
 + \sigma \| \ubar \|_{1,\sigma,H}^2 + \sigma \| \pa_t \ubar \|_{0,\sigma,H}^2.
 \label{eq:abcarl}
\end{align}
Now, on $L$, using (\ref{eq:alpharel}) we have
\begin{align*}
(\pa_t + \omega \cdot \nabla - (a+ \omega \cdot b))(\alpha - \alphaacu)
& = - (\pa_t + \omega \cdot \nabla - (a+ \omega \cdot b)) \alphaacu
\nn
\\
& = - (\pa_t + \omega \cdot \nabla - (\aacu+ \omega \cdot \bacu)) \alphaacu
+ (\abar + \omega \cdot \bbar) \alphaacu
\nn
\\
& = (\abar + \omega \cdot \bbar) \alphaacu.
\end{align*}
Since $\alphaacu$ is positive and bounded away from zero we have
\[
|\abar + \omega \cdot \bbar| \cleq |(\pa_t + \omega \cdot \nabla - (a+ \omega \cdot b))(\alpha - \alphaacu)|.
\]
Using this in (\ref{eq:abcarl}) we obtain
\begin{align*}
\sigma \| \abar + \omega \cdot \bbar \|^2_{0,\sigma, L_\tau}  
& \cleq \|[\abar, \bbar] \|_{0,\sigma,Q_\tau}^2 
 + \sigma \| \ubar \|_{1,\sigma,H_\tau}^2  + \sigma \| \pa_t \ubar \|_{0,\sigma,H_\tau}^2.
\end{align*}

Integrating this w.r.t $\tau$ over $[-1,T+1]$ and repeating the argument in the proof of Theorem
\ref{thm:qstability}, we obtain
\beqn
\sigma \| \abar + \omega \cdot \bbar \|^2_{0, \sigma, \Bbar \times [0,T]}
\cleq \|[ \abar, \bbar]\|^2_{0, \sigma, \Bbar \times [0,T]} + 
\sigma \int_{-1}^{T+1} \| \ubar \|_{1,\sigma,H_{\omega,\tau}}^2  +  \| \ubar_t \|_{0,\sigma,H_{\omega, \tau}}^2
\, d \tau.
\eeqn
Noting that 
\[
2 \abar = (\abar + e^n \cdot \bbar) + (\abar - e^n \cdot \bbar)
\]
and
\[
e^i \cdot \bbar = (\abar + e^i \cdot \bbar) - \abar
\]
we obtain
\begin{align*}
\sigma \| [\abar, \bbar] \|^2_{0, \sigma, \Bbar \times [0,T]}
& \cleq \|[ \abar, \bbar]\|^2_{0, \sigma, \Bbar \times [0,T]} + 
\sigma 
\sum_\omega
\int_{-1}^{T+1} \| \ubar \|_{1,\sigma,H_{\omega, \tau}}^2  +  \| \ubar_t \|_{0,\sigma,H_{\omega, \tau}}^2
\, d \tau,
\end{align*}
where $\omega$ takes the values $-e^n$ and $e^1, \cdots, e^n$.
The theorem follows if we choose $\sigma$ large enough.

%
%

\section{Proof of Theorem \ref{thm:abcunique}}

The proof proceeds as in the proofs of Theorems \ref{thm:qstability} and \ref{thm:abstability} but using
both the $U$ and the $V$ solution. However, we need to
add an idea from \cite{MPS2020} to separate $c$ from $a,b$.

We define 
\[
\phi(x,t) = -\int_{-\infty}^0 (a + e^n \cdot b)(x + s e^n, t+s) \, ds,
\qquad
\acute{\phi}(x,t) = -\int_{-\infty}^0 (\aacu + e^n \cdot \bacu)(x + s e^n, t+s) \, ds;
\] 
we are given that
\begin{align*}
[u, u_t](\cdot, T; \omega, \tau) = [\uacu, \uacu_t](\cdot, T; \omega, \tau)
& \qquad \text{on } H_{\omega, \tau}, ~~ \forall \tau \in [-1, T+1],  ~\omega = e^i, i=1, \cdots, n,
\\
[v, v_t](\cdot, T; e^n, \tau) = [\vacu, \vacu_t](\cdot, T; e^n, \tau)
& \qquad \text{on } H_{\omega, \tau}, ~~ \forall \tau \in [-1, T+1],
\\
[\phi, \phi_t](\cdot,T) = [\phiacu, \phiacu_t](\cdot,T), & \qquad \text{on} ~~\R^n.
\end{align*}
Hence
\begin{align*}
[e^\phi u, (e^ \phi u)_t](\cdot, T; \omega, \tau) = 
[e^\phiacu\uacu , (e^\phiacu \uacu)_t](\cdot, T; \omega, \tau)
& \qquad \text{on } H_{\omega, \tau} ~~ \forall \tau \in [-1, T+1], ~ \omega = e^i, i=1, \cdots, n
\\
[ e^\phi v, (e^\phi v)_t](\cdot, T; e^n, \tau) = 
[e^\phiacu\vacu, (e^\phiacu \vacu)_t](\cdot, T; e^n, \tau)
& \qquad \text{on } H_{\omega, \tau}, ~~ \forall \tau \in [-1, T+1].
\end{align*}
The two sides correspond to the data
for the coefficients $[a+\phi_t, b + \nabla \phi, c]$ and $[\aacu + \phiacu_t, \bacu + \nabla \phiacu, \cacu]$ 
so we work with this
new set of coefficients. What we gain from this new set of coefficients is that
\begin{align*}
 ( (a+\phi_t)+ e^n \cdot (b + \nabla \phi) )(x,t) 
& = (a+ e^n \cdot b)(x,t) + (\pa_t + e^n \cdot \nabla) \phi(x,t)
\\
& = (a+ e^n \cdot b)(x,t)  -  (\pa_t + e^n \cdot \nabla) \int_{-\infty}^0 (a+ e^n \cdot b)(x+se_n, t+s) \, ds
\\
& = (a+ e^n \cdot b)(x,t) - \int_{-\infty}^0  \frac{d}{ds}  ( (a+ e^n \cdot b)(x+se_n, t+s) ) \, ds
\\
& =0.
\end{align*}
Further, $[a,b]$ and $[a + \phi_t, b + \nabla \phi]$ have the same curl. So to prove our theorem
it is enough to show that if we have $[a,b,c]$ and $[\aacu, \bacu , \cacu]$ such that
\begin{align*}
[u, u_t](\cdot, T; \omega, \tau) = [\uacu, \uacu_t](\cdot, T; \omega, \tau)
& \qquad \text{on } H_{\omega, \tau}, ~~ \forall \tau \in [-1, T+1],  ~\omega = e^i, i=1, \cdots, n,
\\
[v, v_t](\cdot, T; e^n, \tau) = [\vacu, \vacu_t](\cdot, T; e^n, \tau)
& \qquad \text{on } H_{\omega, \tau}, ~~ \forall \tau \in [-1, T+1],
\end{align*}
and
\[
a + e^n \cdot b =0, ~~~ \aacu + e^n \cdot \bacu = 0 \qquad \text{on  } \R^n \times (-\infty, T]
\]
then 
\[
[a,b,c] = [\aacu, \bacu, \cacu];
\]
actually we show
\[
[a,b,q] = [\aacu, \bacu, \qacu]
\]
which then implies $c=\cacu$.

{\em Summarizing, we are given that 
\begin{align}
[\ubar, \ubar_t](\cdot, T; \omega, \tau) = 0
& \qquad \text{on } H_{\omega, \tau}, ~~ \forall \tau \in [-1, T+1],  ~\omega = e^i, i=1, \cdots, n,
\label{eq:ubarzero}
\\
[\vbar, \vbar_t](\cdot, T; e^n, \tau) = 0
& \qquad \text{on } H_{\omega, \tau}, ~~ \forall \tau \in [-1, T+1],
\label{eq:vbarzero}
\end{align}
and
\beqn
a + e^n \cdot b =0, ~~~ \aacu + e^n \cdot \bacu = 0 \qquad \text{on  }  \R^n \times (-\infty, T].
\label{eq:abeb0}
\eeqn
We have to show that
\[
[a,b,q] = [\aacu, \bacu, \qacu].
\]
}
Note that the supports of the new $a,b,c, \aacu, \bacu, \cacu$ may not be in $\Bbar \times [0,T]$
but will still be in $\Dbar \times [0,T]$ for some bounded region $D$ in $\R^n$.

Using (\ref{eq:ude}), (\ref{eq:ucc}) and its analogs for $\aacu, \bacu, \cacu$, the function
$\ubar$ satisfies
\begin{align*}
\L \ubar  = 2 \abar \uacu_t - 2 \bbar \cdot \nabla \uacu - \qbar \uacu,
& \qquad \text{on } Q,
\\
\ubar =0, & \qquad t\ll 0,
\\
\ubar  = \alpha - \alphaacu, & \qquad \text{on } L.
\end{align*}
Repeating the argument in the proof of Theorem \ref{thm:abstability}, the only difference being
that $\L \ubar$ now has a $\qbar$ term on the RHS and that (\ref{eq:ubarzero}) holds, one obtains
\beqn
\sigma \|[\abar, \bbar]\|_{0,\sigma, \Dbar \times [0,T]} \cleq 
\|[\abar, \bbar, \qbar]\|_{0,\sigma, \Dbar \times [0,T]}.
\label{eq:abarbbar9}
\eeqn

Next, we take $\omega = e^n$ and we suppress writing the explicit dependence on $e^n$.
Using (\ref{eq:vde}), (\ref{eq:vic}) and its analogs for $\aacu, \bacu, \cacu$, the function
$\vbar$ satisfies
\begin{align*}
\L \vbar  = 2 \abar \vacu_t - 2 \bbar \cdot \nabla \vacu - \qbar \vacu
& \qquad \text{on } Q,
\\
\vbar =0, & \qquad t\ll 0.
\end{align*}
Applying the Carleman estimate in Proposition \ref{prop:carl} to $\vbar$ in the region $Q$ and noting
(\ref{eq:vbarzero}), we have
\begin{align}
\sigma \| \vbar\|_{1,\sigma,L_\tau}^2 
&  \cleq \| \L \vbar \|_{0,\sigma,Q_\tau}^2
 \cleq \|[\abar, \bbar, \qbar]\|_{0,\sigma,Q_\tau}^2.
 \label{eq:vbar6}
\end{align}
In this estimate $\omega=e_n$ and from our discussion above we know that $\alpha =1$ 
and $\alphaacu=1$ in this case. So, on $L$, using (\ref{eq:vcc}) and its equivalent
for $[\aacu, \bacu, \cacu]$, we have
\begin{align*}
2(\pa_t + \omega \cdot \nabla - (a+ \omega \cdot b))(\vbar)
& = 2(\pa_t + \omega \cdot \nabla - (a+ \omega \cdot b))(v - \vacu)
\\
& = - \L \alpha - 2(\pa_t + \omega \cdot \nabla - (\aacu+ \omega \cdot \bacu))\vacu
+ (\abar + \omega \cdot \bbar) \vacu
\\
& =  - \L \alpha + \Lacu \alphaacu + (\abar + \omega \cdot \bbar) \vacu
\\
& = - \qbar + (\abar + \omega \cdot \bbar) \vacu
\\
& = - \qbar.
\end{align*}
Using this in (\ref{eq:vbar6}) we obtain
\beqn
\sigma \| \qbar \|_{0,\sigma, L_{e_n,\tau}}^2
\cleq \|[\abar, \bbar, \qbar]\|_{0,\sigma,Q_{e^n,\tau}}^2.
\label{eq:vbar7}
\eeqn
Integrating this over $\tau \in [-1, T+1]$ and using the arguments used in the proofs of the earlier theorems
we obtain
\[
\sigma \|\qbar \|_{0,\sigma, \Dbar \times [0,T]} \cleq 
\|[\abar, \bbar, \qbar]\|_{0,\sigma, \Dbar \times [0,T]}.
\]
Combining this with (\ref{eq:abarbbar9}) we obtain
\[
\sigma \| [\abar, \bbar, \qbar] \|_{0,\sigma, \Dbar \times [0,T]} \cleq 
\|[\abar, \bbar, \qbar]\|_{0,\sigma, \Dbar \times [0,T]},
\]
so taking $\sigma$ large enough we obtain $\abar =0, \bbar=0, \qbar=0$; hence
$(a,b) = (\aacu, \bacu)$ and $c = \cacu$. However these $a,b, \aacu, \bacu$ are the $\phi$ modified versions
of the old $a,b, \aacu, \bacu$ so we obtain
\[
d \left ( a dt + \sum_{i=1}^n b^i dx^i \right ) = d \left ( \aacu dt + \sum_{i=1}^n \bacu^i dx^i \right ) 
\]
for the older $a,b, \aacu, \bacu$. Of course we have already shown $c= \cacu$.


\section{Proof of Theorem \ref{thm:abcstability}}

From the introduction we know that if $u,v$ are the solutions associated with the coefficients $[a,b,c]$ then
$e^\psi u, e^{\psi} v$ are the solutions associated with the coefficients $[a + \psi_t, b + \nabla \psi, c]$.
Further, using $|e^s -1| \leq e^M |s|$ for all $s \in [-M,M]$ we have
\begin{align*}
| e^\psi w - e^{\psiacu} \wacu| 
& \leq | e^\psi w - e^{\psi} \wacu| + | e^{\psi} \wacu - e^{\psiacu} \wacu|
\\
& \cleq |w - \wacu| + | e^{\psi - \psiacu}-1|
\\
& \cleq |w - \wacu| + |\psi - \psiacu|.
\end{align*} 

Similar estimates hold for the first and second order derivatives of $e^\psi w - e^{\psiacu} \wacu$. Further $[a,b]$ and $[a + \psi_t, b + \nabla \psi]$
have the same curl so we may assume we are working with the coefficients 
$[a + \psi_t, b + \nabla \psi, c]$.
Now
\[
c - (a + \psi_t)_t + \nabla \cdot (b + \nabla \psi)
= c -a_t + \nabla \cdot b - \Box \psi =0.
\]
So it is enough to prove Theorem \ref{thm:abcstability} with the assumption that
\beqn
c - a_t + \nabla \cdot b =0, \qquad \cacu - \aacu_t + \nabla \cdot \bacu =0;
\label{eq:catnb}
\eeqn
note this also implies $\psi=0, ~ \psiacu=0$.

Given the unit vector $\omega$, we define the orthogonal decompositions
\[
\nabla = \no + \nop, \qquad b = \bo + \bop
\]
where
\[
\no := \omega (\omega \cdot \nabla), 
\qquad 
\nop := \nabla - \omega (\omega \cdot \nabla),
\qquad \bo := (\omega \cdot b) \omega, \qquad \bop := b - (\omega \cdot b) \omega.
\]
Note that
\[
\no \cdot \nop = 0 = \nop \cdot \no, \qquad \omega \cdot \nop =0,
\qquad
\bo \cdot \nop =0 = \nop \cdot \bo,
\qquad 
\bop  \cdot \no =0 = \no \cdot \bop.
\]
We obtain some intermediate estimates and, for convenience, {\em temporarily we suppress the dependence on $\tau$}.

\noindent
\underline{Estimate from the $U$ solution.}

Using (\ref{eq:catnb}), we have
\[
\L = \Box - 2 a \pa_t + 2 b \cdot \nabla + a^2 - b^2,
\qquad 
\Lacu = \Box - 2 \aacu \pa_t + 2 \bacu \cdot \nabla + \aacu^2 - \bacu^2;
\]
hence, from  (\ref{eq:ude}), (\ref{eq:ucc}), we have
\begin{align}
\L \ubar  = 2 \abar \uacu_t - 2 \bbar \cdot \nabla \uacu + ((b+\bacu) \bbar - (a+\aacu) \abar) \uacu,
& \qquad \text{on } Q_\omega,
\label{eq:Lubarde}
\\
\ubar =0, & \qquad t\ll 0,
\label{eq:Lubaric}
\\
\ubar  = \alpha - \alphaacu, & \qquad \text{on } L_\omega.
\label{eq:Lubarcc}
\end{align}
So Proposition \ref{prop:carl} applied to $\ubar$ in the region $Q_\omega$ gives us
\begin{align}
\| \ubar \|^2_{1,\sigma,Q_\omega} +
\|\alpha - \alphaacu\|^2_{1,\sigma, L_\omega}
 \cleq \frac{1}{\sigma}\|[\abar, \bbar] \|_{0,\sigma,Q_\omega}^2
 + \| \ubar \|_{1,\sigma,H_\omega}^2  + \| \pa_t \ubar \|_{0,\sigma,H_\omega}^2,
 \label{eq:uest1}
\end{align}

Next, we obtain higher order estimates by differentiating (\ref{eq:Lubarde}) - (\ref{eq:Lubarcc}), keeping in mind that 
$\nop$ and $\pa_t + \omega \cdot \nabla$ span the tangent space to $L_\omega$.

We have
\begin{align*}
\L (\nop \ubar)  = \nop( 2 \abar \uacu_t - 2 \bbar \cdot \nabla \uacu + ((b+\bacu) \bbar - (a+\aacu) \abar)\uacu )
+ [\L, \nop] \ubar,
& \qquad \text{on } Q_\omega
\\
\nop \ubar =0, & \qquad t \ll 0
\\
\nop \ubar  = \nop( \alpha - \alphaacu),
&  \qquad \text{on } L_\omega,
\end{align*}
so, in $Q_\omega$,
\begin{align*}
|\L ( \nop \ubar) | \cleq | [\abar, \bbar, \nabla \abar, \nabla \bbar, \abar_t, \bbar_t]| + |[\ubar, \nabla \ubar, \pa_t \ubar]|.
\end{align*}
Hence applying Proposition (\ref{prop:carl}) to $ \nop \ubar$ we obtain%
\begin{align}
\| \nop (\alpha - \alphaacu) \|_{1,\sigma,L_\omega}^2 
& \cleq \frac{1}{\sigma}\|[\abar, \bbar, \nabla \abar, \nabla \bbar, \abar_t, \bbar_t] \|_{0,\sigma,Q_\omega}^2 
+ \frac{1}{\sigma}
\| \ubar\|_{1,\sigma,Q_\omega}^2
\nn
\\
& \qquad
 + \| \nop \ubar \|_{1,\sigma,H_\omega}^2 
 + \| \pa_t \nop \ubar \|_{0,\sigma,H_\omega}^2.
 \label{eq:tttemp1}
\end{align}
%
We repeat the argument used to obtain (\ref{eq:tttemp1})
with differentiation w.r.t $\pa_t + \omega \cdot \nabla$ replacing differentiation
w.r.t $\nop$. Noting that $\pa_t + \omega \cdot \nabla$ is also tangential 
to $L_\omega$, we obtain
\begin{align}
 \| (\pa_t + \omega \cdot \nabla) (\alpha - \alphaacu) \|_{1,\sigma,L_\omega}^2 
& \cleq \frac{1}{\sigma} \|[\abar, \bbar,\nabla \abar, \nabla \bbar, \abar_t, \bbar_t] \|_{0,\sigma,Q_\omega}^2 
+ \frac{1}{\sigma}
\| \ubar\|_{1,\sigma,Q_\omega}^2
\nn\\
& \qquad 
 +  \| [\nabla \ubar, \pa_t \ubar] \|_{1,\sigma,H_\omega}^2 
 +  \| \pa_t^2 \ubar \|_{0,\sigma,H_\omega}^2.
 \label{eq:tttemp2}
\end{align}

Using (\ref{eq:uest1}), (\ref{eq:tttemp1}), (\ref{eq:tttemp2}), for $\sigma$ large enough, we obtain
\begin{align}
\| (\pa_t + \omega \cdot \nabla)& (\alpha - \alphaacu) \|_{1,\sigma,L_\omega}^2 
 +
\| \nop (\alpha - \alphaacu) \|_{1,\sigma,L_\omega}^2 
+
\| \alpha - \alphaacu\|^2_{1,\sigma, L_\omega}
\nn
\\
& \cleq \frac{1}{\sigma} \|[\abar, \bbar,\nabla \abar, \nabla \bbar, \abar_t, \bbar_t]\|^2_{0,\sigma, Q_\omega}
+ \|[\ubar, \nabla \ubar, \pa_t \ubar] \|^2_{1,\sigma,H_\omega}
+ \|\pa_t^2 \ubar\|^2_{0,\sigma, H_\omega}.
\label{eq:newjest}
\end{align}

We use (\ref{eq:newjest}) to estimate $\abar, \bbar$.
From (\ref{eq:alpharel})
\begin{align}
(\pa_t + \omega \cdot \nabla - (a+ \omega \cdot b))(\alpha - \alphaacu)
& = - (\pa_t + \omega \cdot \nabla - (a+ \omega \cdot b)) \alphaacu
\nn
\\
& = - (\pa_t + \omega \cdot \nabla - (\aacu+ \omega \cdot \bacu)) \alphaacu
+ (\abar + \omega \cdot \bbar) \alphaacu
\nn
\\
& = (\abar + \omega \cdot \bbar) \alphaacu,
\label{eq:Daaprime}
\end{align}
and $\alphaacu$ is positive and bounded away from zero. Hence
\beqn
|\abar + \omega \cdot \bbar| \cleq |(\pa_t + \omega \cdot \nabla)(\alpha - \alphaacu)|
+ |\alpha - \alphaacu|.
\label{eq:jj1}
\eeqn
Differentiating (\ref{eq:Daaprime}) w.r.t $\nop$ and noting that $\nop$ commutes with
$\pa_t + \omega \cdot \nabla$, we obtain
\begin{align}
|\nop (\abar + \omega \cdot \bbar)|
\cleq |\abar + \omega \cdot \bbar|
+ |\nop(\pa_t + \omega \cdot \nabla)(\alpha - \alphaacu)|
+ |\nop (\alpha - \alphaacu)| + |\alpha - \alphaacu|.
\label{eq:jj2}
\end{align}
Differentiating (\ref{eq:Daaprime}) w.r.t $\pa_t + \omega \cdot \nabla$, we obtain
\begin{align}
|(\pa_t + \omega \cdot \nabla) (\abar + \omega \cdot \bbar)|
& \cleq |\abar + \omega \cdot \bbar|
+ |(\pa_t + \omega \cdot \nabla)^2(\alpha - \alphaacu)|
+ |(\pa_t + \omega \cdot \nabla) (\alpha - \alphaacu)|
\nn\\
& \qquad + |\alpha - \alphaacu|.
\label{eq:jj3}
\end{align}
Since $\nop$ and $\pa_t + \omega \cdot \nabla$ are tangential to $L_\omega$, using (\ref{eq:jj1}), (\ref{eq:jj2}), 
(\ref{eq:jj3}) and (\ref{eq:newjest}), we conclude
\begin{align}
 \| [(\pa_t + &  \omega \cdot \nabla) (\abar + \omega \cdot \bbar), \,  \nop(\abar + \omega \cdot \bbar), \,
 \abar + \omega \cdot \bbar ]\|_{0,\sigma,L_\omega}^2
+ \| [\nop (\alpha - \alphaacu), \alpha - \alphaacu] \|_{1,\sigma,L_\omega}^2 
\nn
\\
& \qquad \cleq 
\frac{1}{\sigma} \|[\abar, \bbar, \nabla \abar, \nabla \bbar, \abar_t, \bbar_t] \|_{0,\sigma,Q_\omega}^2 
 +  \| [\ubar, \nabla \ubar, \pa_t \ubar] \|_{1,\sigma,H_\omega}^2 
 + \|  \ubar_{tt} \|_{0,\sigma,H_\omega}^2,
 \label{eq:ufinal}
\end{align}
for large enough $\sigma$.
%
%

\noindent
\underline{Estimate from the $V$ solution.}

Using (\ref{eq:vde}), (\ref{eq:vic}) and its version for $\aacu, \bacu, \cacu$, the function $\vbar$ satisfies
\begin{align*}
\L \vbar =  2 \abar \vacu_t - 2 \bbar \cdot \nabla \vacu + ((b+\bacu) \bbar - (a+\aacu) \abar)\vacu, &  \qquad \text{on}~~ Q_\omega
\\
\vbar =0, & \qquad t\ll 0.
\end{align*}
Hence, applying Proposition (\ref{prop:carl}) to $\vbar$ over the region $Q_\omega$,
we obtain
\beqn
\| \vbar\|_{1,\sigma,L_\omega}^2
\cleq \frac{1}{\sigma}\| [\abar, \bbar]\|_{0,\sigma,Q_\omega}^2 
+ \|\vbar\|_{1,\sigma,H_\omega}^2 + \|\pa_t \vbar\|_{0,\sigma,H_\omega}^2.
\label{eq:vLab}
\eeqn
On $L_\omega$,
\[
2(\pa_t + \omega \cdot \nabla - (a + \omega \cdot b))v = -  \L \alpha,
\qquad
2(\pa_t + \omega \cdot \nabla - (\aacu + \omega \cdot \bacu))\vacu = -  \Lacu \alphaacu,
\]
hence
\begin{align*}
2(\pa_t + \omega \cdot \nabla - (a+ \omega \cdot b))(\vbar)
& = 2(\pa_t + \omega \cdot \nabla - (a+ \omega \cdot b))(v - \vacu)
\\
& = - \L \alpha - 2(\pa_t + \omega \cdot \nabla - (\aacu+ \omega \cdot \bacu))\vacu
+ 2(\abar + \omega \cdot \bbar) \vacu
\\
& =  - \L \alpha + \Lacu \alphaacu + 2(\abar + \omega \cdot \bbar) \vacu
\end{align*}
implying
\[
| (\pa_t + \omega \cdot \nabla - (a+ \omega \cdot b))(\vbar)| \cgeq | \L \alpha - \Lacu \alphaacu| - |[\abar, \bbar]|
\]
which used in (\ref{eq:vLab}) gives us
\beqn
\|\L \alpha - \Lacu \alphaacu\|^2_{0,\sigma,L_\omega} 
\cleq 
\|[\abar, \bbar]\|^2_{0,\sigma, L_\omega}
+ \frac{1}{\sigma}
\| [\abar, \bbar]\|_{0,\sigma,Q_\omega}^2 
+ \|\vbar\|_{1,\sigma,H_\omega}^2 +  \|\pa_t \vbar\|_{0,\sigma,H_\omega}^2.
\label{eq:Laaprime}
\eeqn

We need a different representation for $\L \alpha - \Lacu \alphaacu$. We claim
\begin{align}
	 \L \alpha  = \alpha ( \pa_t - \omega \cdot \nabla ) ( a + \omega \cdot b ) - \left (  \nop^2 - 2 \bop \cdot \nop  + \bop^2 \right ) \alpha,
	\label{eq:newLalpha}
\end{align}
provided $c-a_t + \nabla \cdot b=0$ - we postpone the proof of (\ref{eq:newLalpha}) to the end of this section.
Then 
\begin{align*}
	\L \alpha - \Lacu \alphaacu 
	&  = \alpha \left ( \pa_t - \omega \cdot \nabla   \right ) ( \abar + \omega \cdot \bbar ) + ( \alpha -\alphaacu) \left ( \pa_t - \omega \cdot \nabla   \right ) ( \aacu + \omega \cdot \bacu )
	- \nop^2 (\alpha - \alphaacu)  \\
	& \qquad 
	+ 2 \bop \cdot \nop (\alpha - \alphaacu) - \bop^2 (\alpha - \alphaacu) 
	+ 2 \bbar_\omega^\perp \alphaacu - \bbar_\omega^\perp \cdot (b+ \bacu)_\omega^\perp \alphaacu.
\end{align*}
Using this and that $\alpha$ is bounded away from zero, we have
\[
|\L \alpha - \Lacu \alphaacu | \cgeq 
\left | \left ( \pa_t - \omega \cdot \nabla   \right ) ( \abar + \omega \cdot \bbar ) \right | 
- | \nop^2 (\alpha - \alphaacu)| - |\nop (\alpha - \alphaacu)| - |\alpha - \alphaacu| - |\bbar|,
\]
which used in (\ref{eq:Laaprime}) gives us
\begin{align}
	 \| \left ( \pa_t - \omega \cdot \nabla   \right ) ( \abar + \omega \cdot \bbar ) \|^2_{0,\sigma,L_\omega}
	& \cleq  \|[\abar,\bbar]\|_{0,\sigma,L_\omega}^2 +
	\| \nop^2 (\alpha - \alphaacu) \|_{0,\sigma,L_\omega}^2
	+ \|\nop (\alpha - \alphaacu) \|_{0,\sigma,L_\omega}^2 
	\nn
	\\
	& \qquad + \|\alpha - \alphaacu\|_{0,\sigma,L_\omega}^2  + \frac{1}{\sigma} \| [\abar, \bbar]\|_{0,\sigma,Q_\omega}^2 
	+  \|\vbar\|_{1,\sigma,H_\omega}^2 +  \|\pa_t \vbar\|_{0,\sigma,H_\omega}^2.
	\label{eq:finalvest}
\end{align}
%
%
\noindent
\underline{Combining the $U$, $V$ estimates.}

 Multiplying the $V$ based estimate (\ref{eq:finalvest}) by a small (independent of $\sigma$) $\ep$ in $(0,1)$ and adding it to the combined $U$ based estimate (\ref{eq:ufinal}), we obtain 
\begin{align}\label{gen est}
	\ep \| \left ( \pa_t - \omega \cdot \nabla   \right ) &  ( \abar + \omega \cdot \bbar ) \|^2_{0,\sigma,L_{\omega, \tau}}
	+  \| \abar + \omega \cdot \bbar \|_{0,\sigma,L_{\omega,\tau}}^2 \nn
	\\
	& +  \| \nop(\abar + \omega \cdot \bbar) \|_{0,\sigma,L_{\omega, \tau}}^2
	+  \| (\pa_t + \omega \cdot \nabla) (\abar + \omega \cdot \bbar) \|_{0,\sigma,L_{\omega, \tau}} ^2   \nn 
	\\
	& \qquad \qquad \cleq \ep \| [\abar, \bbar] \|_{0,\sigma, L_{\omega, \tau}}^2
	+ \frac{1}{\sigma} \|[\abar, \bbar, \nabla \abar, \nabla \bbar, \abar_t, \bbar_t] \|_{0,\sigma,Q_{\omega,\tau}}^2    
	+ \text{data}_{\omega,\tau,\sigma}
\end{align}
where
\[
\text{data}_{\omega,\tau, \sigma} =  \| [\nabla \ubar, \ubar_t, \ubar] \|_{1,\sigma,H_{\omega,\tau}}^2 
+ \| \ubar_{tt}\|^2_{0,\sigma,H_{\omega,\tau}}
+  \|\vbar\|_{1,\sigma,H_{\omega,\tau}}^2 + \| \vbar_t \|^2_{0, \sigma, H_{\omega,\tau}}.
\]
Expanding the LHS of (\ref{gen est}) we get (noting that $\epsilon<1$)
\begin{align*}
    \| \abar +  \omega \cdot \bbar & \|_{0,\sigma,L_{\omega,\tau}}^2  + \ep \| \nop(\abar + \omega \cdot \bbar) \|_{0,\sigma,L_{\omega, \tau}}^2\\
    & + \ep  \| \pa_t (\abar + \omega \cdot \bbar) \|_{0,\sigma,L_{\omega, \tau}}^2 + \ep \| (\omega \cdot \nabla) (\abar + \omega \cdot \bbar) \|_{0,\sigma,L_{\omega, \tau}}^2\\
    & \qquad \qquad \cleq \ep \| [\abar, \bbar] \|_{0,\sigma, L_{\omega, \tau}}^2
	+ \frac{1}{\sigma} \|[\abar, \bbar, \nabla \abar, \nabla \bbar, \abar_t, \bbar_t] \|_{0,\sigma,Q_{\omega,\tau}}^2    
	+ \text{data}_{\omega,\tau,\sigma}
\end{align*}
which immediately gives
\begin{align*}
    \| \abar +  \omega \cdot \bbar  \|_{0,\sigma,L_{\omega,\tau}}^2  +\  & \ep \| \nabla_{x,t} (\abar + \omega \cdot \bbar) \|_{0,\sigma,L_{\omega, \tau}}^2 \\
    & \cleq \ep \| [\abar, \bbar] \|_{0,\sigma, L_{\omega, \tau}}^2
	+ \frac{1}{\sigma} \|[\abar, \bbar, \nabla \abar, \nabla \bbar, \abar_t, \bbar_t] \|_{0,\sigma,Q_{\omega,\tau}}^2    
	+ \text{data}_{\omega,\tau,\sigma}
\end{align*} 

Integrating this w.r.t $\tau$ over $[-1,T+1]$ and repeating the argument used at the end of the proof of Theorem \ref{thm:qstability}  we obtain
\begin{align*}
	\| \abar +  \omega \cdot \bbar  &  \|_{0,\sigma,\R^n \times [0,T]}^2  +\ \ep \| \nabla_{x,t} (\abar + \omega \cdot \bbar) \|_{0,\sigma,\R^n \times [0,T]}^2 \\
	& \cleq \ep  \| [\abar,\bbar] \|_{0,\sigma, \R^n \times [0,T]}^2
	+ \frac{1}{\sigma} \|[\abar, \bbar, \nabla \abar, \nabla \bbar, \abar_t, \bbar_t]\|_{0,\sigma, \R^n \times [0,T]}^2 
	+  \int_{-1}^{T+1} \text{data}_{\omega,\tau,\sigma} \, d \tau.
\end{align*}

All norms below are  $\| \cdot \|_{0,\sigma, \R^n \times [0,T]}$ unless noted otherwise. To complete the proof, we repeat the argument in the proof of Theorem 1.4. That is, we vary $\omega$ in the set $\{\pm e^n, e^i \ ; 1 \le i \le n-1 \}$ and finally obtain
\begin{align*}
	\| [\abar,\bbar] \|^2 + \ep \| [\nabla_{x,t}\abar, \nabla_{x,t}\bbar]  \|^2 \cleq \ep  \| [\abar,\bbar] \|^2
	+ \frac{1}{\sigma} \|[\nabla_{x,t}\abar, \nabla_{x,t}\bbar] \|^2
	+  \int_{-1}^{T+1} \text{data}_{\pm \omega,\tau,\sigma} \, d \tau.
\end{align*}
So taking $\ep$ small enough and then fixing a $\sigma$ large enough
\[
\|[\abar, \bbar, \nabla_{x,t} \abar, \nabla_{x,t} \bbar] \|^2
\cleq \sum_{i=1}^n  \int_{-1}^{T+1} \text{data}_{\pm e_i, \tau, \sigma} \, d \tau.
\]
Since $\cbar = \abar_t - \nabla \cdot \bbar$, we conclude
\[
\|[\abar, \bbar, \cbar, \nabla_{x,t} \abar, \nabla_{x,t} \bbar] \|^2
\cleq \sum_{i=1}^n  \int_{-1}^{T+1} \text{data}_{\pm e_i, \tau, \sigma} \, d \tau,
\]
for the fixed large enough $\sigma$.

For a fixed $\sigma$, on a compact set, the weighted and unweighted norms are equivalent, so the theorem is proved. It remains to show (\ref{eq:newLalpha}) 
when $c = a_t - \nabla \cdot b$.

\noindent
\underline{Proof of (\ref{eq:newLalpha}).}

We note that
\begin{align*}
	\L \alpha & 
	= \left ( \pa_t^2 - \Delta - 2 a \pa_t + 2 b \cdot \nabla - a_t + \nabla \cdot b + a^2 - b^2 + c \right ) \alpha\\
	& = \left ( \pa_t^2 - ( \omega \cdot \nabla )^2 - \nop^2 - 2 a \pa_t + 2 b \cdot \nabla + a^2 - b^2 \right ) \alpha
	\\
	& = \left (  ( \pa_t - \omega \cdot \nabla )  ( \pa_t + \omega \cdot \nabla ) - \nop^2 - 2 a \pa_t + 2 b \cdot \nabla + a^2 - b^2 \right ) \alpha.
\end{align*}
Hence using (\ref{eq:alpharel})
\begin{align*}
	\L \alpha 
	& =  ( \pa_t - \omega \cdot \nabla) \left ( ( a + \omega \cdot b ) \alpha \right ) - \left ( \nop^2 + 2 a \pa_t - 2 b \cdot \nabla - a^2 + b^2 \right ) \alpha
	\\
	& = \alpha ( \pa_t - \omega \cdot \nabla)( a + \omega \cdot b ) - \left ( -( a+ \omega \cdot b ) ( \pa_t - \omega \cdot \nabla) + \nop^2 + 2 a \pa_t - 2 b \cdot \nabla - a^2 + b^2 \right ) \alpha,\\
	& = \alpha ( \pa_t - \omega \cdot \nabla) ( a + \omega \cdot b ) - \left ( ( a-  \omega \cdot b ) ( \pa _t + \omega \cdot \nabla ) + \nop^2 - 2 ( b - \omega ( \omega \cdot b ) ) \cdot \nabla - a^2 + b^2  \right ) \alpha\\
	& = \alpha ( \pa_t - \omega \cdot \nabla ) ( a + \omega \cdot b )  - \left ( ( a-  \omega \cdot b ) ( a + \omega \cdot b ) + \nop^2 - 2 ( b - \omega ( \omega \cdot b ) ) \cdot \nabla - a^2 + b^2  \right ) \alpha\\
	& = \alpha ( \pa_t - \omega \cdot \nabla ) ( a + \omega \cdot b ) - \left (  \nop^2 - 2 \bop \cdot \nop      + \bop^2 \right ) \alpha. 
\end{align*}

%


%

\section{The Carleman estimate}\label{sec:carl}

We show that the standard Carleman estimate with boundary terms
holds for the operator $\L_{a,b,c}$ with the Carleman weight $t$ over the region $Q_{\omega, \tau}$.
We need the explicit boundary terms in the proofs of our theorems. Here $a,b^i,c$ are compactly supported
smooth functions on $\R^n \times [0,T]$.
\begin{prop}\label{prop:carl}
If $w(x,t)$ is a compactly supported $C^3$ function on $Q_{\omega, \tau}$ then, for large enough $\sigma$,
we have
\begin{align}
\sigma \int_{Q_{\omega, \tau}} e^{2 \sigma t} ( |\nabla_{x,t} w|^2 + \sigma^2 |w|^2 )
 & + \sigma \int_{L_{\omega,\tau}}e^{2 \sigma t} ( |\nabla_L w|^2 + \sigma^2 |w|^2) 
 \nn
 \\
& \cleq \int_{Q_{\omega, \tau}} e^{2 \sigma t} |\L_{a,b,c} w|^2 
+ \sigma \int_{H_{\omega, \tau}} e^{2 \sigma t} ( |\nabla_{x,t} w|^2 + \sigma^2 |w|^2)
\label{eq:vcarlpos}
\end{align}
with the constant independent of $w$ and $\sigma$. Here $\nabla_L$ is the gradient operator on the plane 
$L_{\omega, \tau}$.
\end{prop}
%
%
\begin{proof}
This proposition could probably be proved by using energy estimates coming from standard multipliers
but we use Carleman estimates since we have already calculated the boundary terms in \cite{rakesh2019fixed} for a 
general situation. Below, we use the notation used for Theorem A.7 in \cite{rakesh2019fixed}.

We appeal to Theorem A.7 of \cite{rakesh2019fixed}. The hypothesis of Theorem  A.7 requires the strong pseudo-convexity of $\phi$ but the proof of Theorem A.7 just needs that the relation (A.9) (in Lemma A.6) holds. 
One can verify that (A.9) holds for the wave operator
and $\phi(x,t) = t$. In fact (A.9) holds because there are no {\em ``$(x,\xi,\sigma) \in \bar{\Omega} \times S$
with $p(x,\xi) - \sigma^2 p(x,\pa \phi)=0$ and $\{p,\phi\}(x,\xi)=0$''} as we show next. We have
$p(x,t,\xi,\tau) = \tau^2 - |\xi|^2$ and $\phi(x,t) =  t$. Hence 
\[
0 = \{p, \phi\}(x,t,\xi,\tau) = p_\tau \phi_t = \pm 2 \tau
\]
and
\[
0 = p(x,t,\xi,\tau) - \sigma^2 p(x,t, \nabla \phi, \phi_t) = \tau^2 - |\xi|^2 - \sigma^2
\]
imply $\tau=0, \xi=0, \sigma=0$, hence there are no points ``$(x,\xi,\sigma) \in \bar{\Omega} \times S$
with $p(x,\xi) - \sigma^2 p(x,\pa \phi)=0$ and $\{p,\phi\}(x,\xi)=0$''. Note $S$ represents the unit sphere.

The proposition will follow from an analysis of the boundary terms in the 
statement of Theorem A.7. The principal part of $\L_{a,b,c}$ is the wave operator and without loss of 
generality we assume that $\tau=0$, $x=(y,z)$ with $y \in \R^{n-1}$, $z \in \R$ and $\omega$ is the unit vector 
in the direction of the positive $z$ axis hence $L_{\omega, \tau}$ is the plane $t=z$.

The boundary term on $t=z$ has been computed in Subsection A.2 in \cite{rakesh2019fixed} and  is given by
\begin{align*}
\frac{1}{\sqrt{2}} \nu_j E_j & = (\phi_t - \phi_z)(u_z+u_t)^2 + (\phi_z + \phi_t) |u_y|^2 - 2(u_z+u_t)(u_y \cdot \phi_y) 
\\
& \qquad - \sigma^2(\phi_z+\phi_t)(|\phi_x|^2 - \phi_t^2)u^2 - (u_z+u_t) g(x,t) u.
\end{align*}
where $u = w e^{\sigma \phi}$ and $g$ is some smooth function independent of $\sigma$. 
Hence, on $t=z$ for $\phi=t$ we have $u = w e^{\sigma t}$ and 
\begin{align*}
\frac{1}{\sqrt{2}} \nu_j E_j & = (u_z+u_t)^2 + |u_y|^2 + \sigma^2u^2 - (u_z+u_t) g(x,t) u
\\
& \geq (u_z+u_t)^2 + |u_y|^2 + \sigma^2u^2 - \frac{1}{2} (u_z+u_t)^2 - k u^2
\\
& = \frac{1}{2} (u_z+u_t)^2 + |u_y|^2 + (\sigma^2 - k) u^2
\qquad \qquad \text{$k$ independent of $\sigma$}
\\
& \cgeq e^{2 \sigma t}  ( (w_z+w_t)^2  + |w_y|^2 + \sigma^2 w^2 )
\qquad \qquad \text{using a standard argument}
\end{align*}
for $\sigma$ large enough.

To get the boundary terms on $t=T$, we again go to the expressions in subsection A.2 on \cite{rakesh2019fixed}
for the wave operator. Here $\nu_x=0$ and $\nu_t = (0,0, \cdots,0,1)$, hence $\nu_j E_j =0$ for
$j=1, \cdots, n$ and
\[
\nu_t E_t = - \phi_t ( |u_x|^2 - u_t^2) + \sigma^2 \phi_t( |\phi_x|^2 - \phi_t^2)u^2 
+ 2u_t ( u_x \cdot \phi_x - u_t \phi_t) + g(x,t) u_t u.
\]
Hence, on $t=T$, for $\phi=t$ we have
\begin{align*}
\nu_t E_t & = - (|u_x|^2 - u_t^2) - \sigma^2 u^2 - 2 u_t^2 + g(x,t) u_t u
\\
& = - (|u_x|^2 + u_t^2) - \sigma^2 u^2 + g(x,t) u_t u
\\
& \cgeq - e^{2\sigma t} ( |\nabla_{x,t} w|^2 + \sigma^2 w^2)
\end{align*}
by a standard argument. 
The proposition now follows from (A.11) of Theorem A.7 in \cite{rakesh2019fixed}.
\end{proof}
%
\section{The forward problems}

\subsection{Proof of Proposition \ref{prop:heaviside}}
The existence, uniqueness and the regularity may be proved in a fashion similar to the proof of Proposition 1.1 in \cite{RS1}. The only part which is new is the progressing wave expansion which we show below. Below, $\L$ will mean $\L_{a,b,c}$.

We seek $U$ in the form
\[
U(\xtot) = u(\xtot)H(\ttxo)
\]
for some function $u(\xtot)$ defined on the region $t \geq \tau + x \cdot \omega$. To describe $u(\xtot)$ in detail, 
we work with the special case when $\tau=0$; the general $\tau$ result will be inferred easily from this special 
case. Below we denote $U(x,t;\omega,0), u(x,t;\omega,0)$ and $\alpha(x;\omega,0)$ by 
$U(x,t;\omega), u(x,t;\omega)$ and $\alpha(x;\omega)$. 

The initial condition and the speed of propagation force 
\[
u(\xto)=1, \qquad \text{when } t\ll 0.
\]
Also, observe that
\begin{align*}
(\pa_t -a) ( f(x,t) F(\txo) ) & = f F'(\txo) + ( (\pa_t -a) f)F(\txo)
\\
(\pa_t-a)^2( f(x,t) F(\txo) ) & = f F''(\txo) + 2( (\pa_t -a) f)F'(\txo) 
\\
& \qquad + ( (\pa_t -a)^2 f) F(\txo)
\\
(\nabla - b) ( f(x,t) F(\txo) ) & = - \omega f F'(\txo) + ( ( \nabla -b)f) F(\txo)
\\
( (\nabla -b)^2 ) F(\txo) ) & =  f F''(\txo) - 2 (\omega \cdot (\nabla - b) f) F'(\txo) 
\\
& \qquad + ( (\nabla -b)^2 f) F(\txo),
\end{align*}
so
\begin{align}
\L ( f(x,t) F(\txo) ) & = 2( f_t +  \omega \cdot \nabla f - (a+ \omega \cdot b)f) F'(\txo) 
+ (\L f) F(\txo).
\label{eq:LF}
\end{align}
Hence
\begin{align*}
\L U & =2 (u_t + \omega \cdot \nabla u - (a + \omega \cdot b) u ) \delta(\txo) + (\L u) H(\txo).
\end{align*}
This forces $\L u =0$ on the region $t \geq x \cdot \omega$ and, on $t = x \cdot \omega$, $u$ must satisfy the
transport equation
\[
(u_t + \omega \cdot \nabla u - (a + \omega  \cdot b) u)(x, x \cdot \omega ; \omega)  =0.
\]

Noting (\ref{eq:alpharel}) and that $\alpha(x,x \cdot \omega;\omega)=1$ for $x \cdot \omega <<0$ 
we see that
$ u(x , x \cdot \omega; \omega)  = \alpha(x, x \cdot \omega; \omega)$.
Hence 
\[
U(x,t;\xi, \tau) = u(\xtot) H(\ttxo)
\]
where $u(\xtot)$ is the solution of the characteristic initial value problem.

\subsection{Proof of Proposition \ref{prop:delta}}
The existence, uniqueness and the regularity may be proved in a fashion similar to the proof of Proposition 1.1 in \cite{RS1}. The only part which is new is the progressing wave expansion which we show below. Below, $\L$ will mean $\L_{a,b,c}$.

We seek $V$ in the form
\[
V(\xtot) = f(x,t; \omega, \tau) \delta(\ttxo) + v(\xtot) H(\ttxo)
\]
with $v(\xtot)$ supported in the region $t \geq \tau + x \cdot \omega$ and, for $t <<0$, we have
 $f(x,t;\omega,\tau)=1$ and $v(\xtot)=0$. There are many choices for 
 $f(x,t;\omega, \tau)$ but a unique choice for 
$f(x,\tau + x \cdot \omega; \omega, \tau)$.
To describe $V(\xtot)$ in detail we work with the special case when $\tau=0$; the general $\tau$ result will be 
inferred easily from this special case. Below we denote $V(x,t;\omega,0), f(x,; \omega,0)$ and 
$v(x,t;\omega,0)$  
by $V(x,t;\omega), f(x;\omega)$ and $v(x,t;\omega)$. 

We seek $V$ in the form
\[
V(\xto) = f(\xto) \delta(\txo) + v(\xto) H(\txo),
\]
hence, using (\ref{eq:LF}), 
\begin{align*}
(\L V)  (x,t;\omega) & =  2 (f_t + \omega \cdot \nabla f  - (a + \omega \cdot b) f)
(x,t; \omega) \delta'(\txo) 
\\
& \qquad + (2v_t + 2\omega \cdot \nabla v - 2(a + \omega \cdot b) v + \L f)
(x, x \cdot \omega;\omega) \delta(\txo)
\\
& \qquad +  (\L v)(x,t;\omega) H(\txo).
\end{align*}
Amongst the many choices for $f$ to zero out the first term in the above expansion of $\L V$,
we choose one for which
\[
f_t +  \omega \cdot \nabla f - (a+ \omega \cdot b)f =0, \qquad \text{on } \R^n \times \R.
\]
Hence we chose $f(x,t;\omega) = \alpha(x,t;\omega)$ so we must now require
\[
Lv =0 \qquad \text{on } t \geq  x \cdot \omega
\]
and, on $t=x \cdot \omega$, $v$ must satisfy the transport equation
\begin{align*}
2(v_t + \omega \cdot \nabla v - (a + \omega  \cdot b) v)(x, x \cdot \omega; \omega) & = 
- (\L \alpha)(x, x \cdot \omega; \omega),
 \qquad x \in \R^n.
\end{align*}

So for a general $\tau$, 
\[
V(x,t;\omega,\tau) = \alpha(x, t;\omega) \delta(\ttxo) + v(\xtot) H(\ttxo) 
\]
where $v(\xtot)$ is the solution of the characteristic IVP.

\section*{Acknowledgments} 
Krishnan is supported in part by India SERB Matrics Grant MTR/2017/000837. Rakesh's work was supported 
by the NSF grant DMS 1908391.

\bibliographystyle{plain}

\end{document}